\documentclass{article}
\usepackage{setspace}
\usepackage[utf8]{inputenc}
\usepackage{amsmath,amssymb,amsfonts,mathtools}
\usepackage{amsthm}
\usepackage{geometry}
\usepackage{cancel}

\usepackage{subcaption}
\usepackage{footnote}
\makesavenoteenv{tabular}
\makesavenoteenv{table}
\usepackage{authblk}
\usepackage{graphicx}
\usepackage{color}
\def\defas{\ensuremath{\mathrel{:=}}}

\DeclareMathOperator{\id}{id}

\def\norm#1{\|  #1 \| }

\def\abs#1{|  #1 | }

\DeclareMathDelimiter{\orbrack}{\mathopen}{operators}{"5D}{largesymbols}{"03}
\DeclareMathDelimiter{\clbrack}{\mathclose}{operators}{"5B}{largesymbols}{"02}
\def\intcc#1{\ensuremath{[#1]}}

\newtheorem{theorem}{Theorem}
\newtheorem{lemma}[theorem]{Lemma}

\title{Euler-Bernoulli beams with contact forces: existence, uniqueness, and numerical solutions\thanks{This work was funded by NSERC DG, CRC, and ERA programs.}}
\author[1]{Mohamed A. Serry}
\affil[1]{Mechanical and Mechatronics Engineering, University of Waterloo, Waterloo, Ontario N2L 3G1, Canada}
\author[1]{Sean D. Peterson}
\author[2]{Jun Liu}
\affil[2]{Applied Mathematics, University of Waterloo, Waterloo, Ontario N2L 3G1, Canada}

\date{}

\providecommand{\keywords}[1]
{
  \small	
  \textbf{\textit{Keywords---}} #1
}

\begin{document}

\maketitle

\begin{abstract}
 In this paper, we investigate the Euler-Bernoulli  fourth-order boundary value problem (BVP) $w^{(4)}=f(x,w)$,   $x\in \intcc{a,b}$, with specified values of $w$ and $w''$ at the end points, where the behaviour of the right-hand side $f$ is motivated by  biomechanical, electromechanical, and structural applications incorporating contact forces. In particular, we consider the case when $f$ is bounded above and monotonically decreasing with respect to its second argument.  First, we prove the existence and uniqueness  of solutions to the   BVP. We then study numerical solutions to the BVP, where we resort to spatial discretization by means of finite difference. Similar to the original continuous-space problem, the discrete problem always possesses a unique solution.  In the case of a piecewise linear instance of $f$, the discrete problem is  an example of the absolute value equation.  We show that solutions to this absolute value equation can be obtained by means of fixed-point iterations, and that     solutions to the absolute value equation converge to  solutions of the continuous BVP. We also illustrate the performance of the  fixed-point iterations through a numerical example.
\end{abstract}
\section*{Introduction}
 In biomechanical applications, especially those related to voice production, contact forces arise naturally. For example, prior to and during voice production, the vocal folds (VFs) typically experience contact forces  as they  touch each other and/or collide.  In general, and due to the emergence of contact forces and the complexity of the air flow interacting with the VFs during phonation, voice-related models  are typically non-smooth.  Despite the non-smoothness, numerical solvers that are designed primarily for smooth systems are often implemented when solving phonation equations. Moreover, the nature of solutions (e.g., existence, uniqueness, smoothness, etc.) to these non-smooth equations is usually dismissed as they tend to be analytically intractable, and the effectiveness of  numerical solvers (i.e., convergence) are tested only empirically. Consequently, a major portion of the numerical studies of voice-related models lack mathematical rigor, with no guarantees  regarding the correctness and accuracy of the reported numerical solutions. Therefore, it is important to  investigate voice-related models mathematically, when feasible, and elucidate the nature of their analytical and numerical solutions.

The VFs are often modelled as elastic structures, with incorporated repulsive contact forces that take place when the folds touch each other, with forces being proportional to the degree of compression (see, e.g., \cite{StoryTitze95,SerrySteppPeterson21,GalindoPetersonErathCastroHillmanZanartu17}). A very simple abstraction of the VFs, which is more suited for scenarios concerning  static  configurations (e.g., posturing)\footnote{VF  configuration prior to phonation is an important aspect in voice production as it affects  VF vibrations and may have consequences regarding vocal health \cite{ZanartuGalindoErathPetersonWodickaHillman14,DejonckereKob09}.}, relies on considering one of the VFs as an Euler-Bernoulli beam that experiences contact spring forces, and assuming the other fold to be an image of the first one with respect to the contact plane; see Figure \ref{fig:Abstraction}.
\begin{figure}
    \centering
    \includegraphics[width=0.6\linewidth,trim={0in 1in 0in 0.5in},clip]{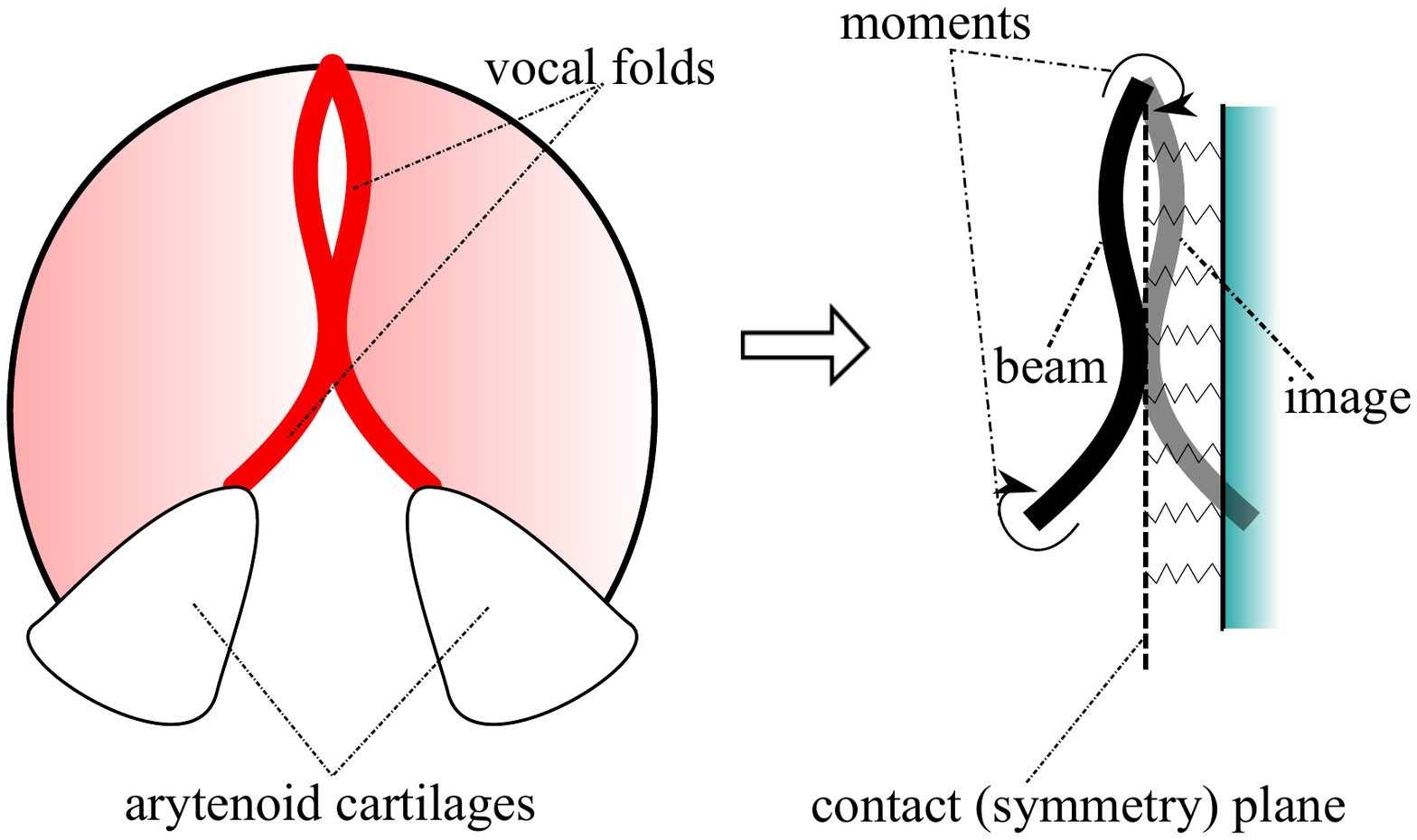}    \caption{Modelling the configuration of the vocal folds (left) using a beam with spring forces resembling contact (right).}
    \label{fig:Abstraction}
\end{figure}
That is, the VF configuration, presented by the deflection function $v$, satisfies the  differential equation $\beta v^{(4)}=F_{s}(s)$, where $\beta$ is the bending stiffness and $F_{s}(s)$ is the contact spring force (per unit length), in addition to complementary boundary conditions (BCs) in terms of displacements and moments at the end points of the beam. Such an abstraction is analytically tractable, allowing rigorous mathematical investigation that we pursue in this work. In fact, such modelling framework  has been adopted to understand curved VF configurations under different laryngeal conditions \cite{SerryAlzamendiZanartuPeterson23b}. In addition to biomechanical applications,  similar beam models with contact forces arise in electromechanical systems, such as  micoswithches  \cite{JensenHuangChowKurabayashi05,McCarthyAdamsMcGruerPotter02}, and civil structural applications, such as beams on elastic foundations \cite{HetenyiHetbenyi46,Barber11,VallabhanDas91,Jones97}. Therefore, due to their relevance in different applications, we  study mathematically  Euler-Bernoulli beams with contact forces in terms of solution existence and uniqueness and the properties of numerical solutions, especially those resulting from finite difference approximations.

The organization of this paper is as follows. Necessary preliminaries and notations are provided in Section \ref{sec:Preliminaries}, problem formulation and associated assumptions are introduced in Section \ref{sec:ProblemFormulation}, relevant findings from the literature are presented in Section \ref{sec:LiteratureReview}, existence and uniqueness results and their proofs are presented in Section \ref{sec:ExistenceUniquenss}, the discretized boundary value problem and the existence and uniqueness of its solutions are discussed in Section \ref{sec:NumericalApproach} with a further consideration of the numerical solutions in the case of a piecewise linear right-hand side in Section \ref{sec:NumericalPWL}, and the study is concluded in Section \ref{sec:Conclusion}.

\section{Preliminaries}
\label{sec:Preliminaries}
Let $\mathbb{R}$, $\mathbb{R}_+$, $\mathbb{Z}$, and $\mathbb{Z}_{+}$ denote
the sets of real numbers, nonnegative real numbers, integers, and
nonnegative integers, respectively, and
$\mathbb{N} = \mathbb{Z}_{+} \setminus \{ 0 \}$.
Let $\intcc{a,b}$ denote
the closed interval with end points $a$ and $b$, and
 $\intcc{a;b}$ stands for its discrete counterpart ($\intcc{a;b} = \intcc{a,b} \cap \mathbb{Z}$). For $x \in \mathbb{R}^{n}$, we define $\abs{x} = (\abs{x_1},\ldots , \abs{x_n})$. Given $x,y\in \mathbb{R}^{n}$, $x \leq y\Leftrightarrow x_{i} \leq y_{i}~\forall i\in \intcc{1;n}$. In this work,  
 $\| \cdot \|_{p}$ denotes any $p$-norm on $\mathbb{R}^{n}$ with $p\in \intcc{1,\infty}$, e.g., $\norm{x}_{\infty}=\max_{i\in \intcc{1;n}}\abs{x_{i}},~x\in \mathbb{R}^{n}$.  Given $A \in \mathbb{R}^{n \times m}$, $\norm{A}_{p}$  denotes the matrix norm induced by $\norm{\cdot}_{p}$ (e.g.,  $\norm{A}_{\infty}=\max_{i\in \intcc{1;n}}\sum_{j=1}^{m}\abs{A_{i,j}}$, $\norm{A}_{1}=\max_{j\in \intcc{1;m}}\sum_{i=1}^{n}\abs{A_{i,j}}$ ). Given  $A\in \mathbb{R}^{n\times n}$,  $\sigma (A)$  denotes the set of eigenvalues  of $A$. If $A\in \mathbb{R}^{n}$ is symmetric ($A=A^{\intercal}$), then   $\lambda_{i}(A)$ denotes the $i^{\mathrm{th}}$ eigenvalue of $A$ (which is real), such that
 $
\lambda_{i}(A)\leq \lambda_{i+1}(A),~i\in \intcc{1;n-1}.
 $
 The identity map defined on $\mathbb{R}^{n}$ is denoted by $\id$, where the dimension will be clear from the context.
Integration of single-valued  functions $f\colon X\rightarrow Y$ is
always understood in the sense of  Lebesgue.  $ C\intcc{a,b}$ denotes  the Banach space of continuous functions $f:\intcc{a,b}\rightarrow \mathbb{R}$ equipped with the uniform norm
$
\norm{f}_{C\intcc{a,b}}\defas\sup_{x\in \intcc{a,b}}|f(x)|
$. Given $p,q\in C\intcc{a,b}$, $p \leq q\Leftrightarrow p(x)\leq q(x) ~\forall x\in \intcc{a,b}$. Let $\mathrm{Lip}_{k}\intcc{a,b}$, where $k\in \mathbb{R}_{+}$, denotes the space of Lipschitz continuous functions $f:\intcc{a,b}\rightarrow \mathbb{R}$, with Lipschitz constants less than or equal to $k$; that is
$
f\in \mathrm{Lip}_{k}\intcc{a,b} \Leftrightarrow \abs{f(x)-f(y)}\leq k\abs{x-y}~\forall x,y\in \intcc{a,b}.
$
Finally, $C^{k}\intcc{a,b}$, where $k\in \mathbb{N}$, denotes the space of $k$-times continuously differentiable functions $f:\intcc{a,b}\rightarrow \mathbb{R}$ (derivatives at $a$ and $b$ are one-sided).

\section{Problem formulation}
\label{sec:ProblemFormulation}
As a generalization of the Euler-Bernoulli model with contact discussed in the introduction, we consider the fourth-order BVP over the compact interval $\intcc{a,b}$ with respect to the  function $w:\intcc{a,b}\rightarrow \mathbb{R}$:
\begin{subequations}
\label{eq:GeneralBVP} 
\begin{align} \label{eq:ODE}
w^{(4)}=f(x,w),&~x\in \intcc{a,b},\\ \label{eq:BCs}
   w(a)=\alpha _{1},~w(b)=\alpha_{2},&~
   w''(a)=\beta_{1},~
   w''(b)=\beta_{2},
   \end{align}
\end{subequations}
where $\alpha_{1},\alpha_{2}, \beta_{1},\beta_{2}\in \mathbb{R}$ are specified boundary values, and $f:\intcc{a,b}\times\mathbb{R}\rightarrow \mathbb{R}$  satisfies the following set of properties:
\begin{subequations} \label{eq:Assumptions}
    \begin{align} \label{eq:ContinuityAssumption}
f~\text{is continuous on } \intcc{a,b}\times\mathbb{R},
\\    
\label{eq:NegativityAssumption}
f(x,y)\leq M~ \forall (x,y)\in \intcc{a,b}\times\mathbb{R}~\text{for some}~M\in \mathbb{R},\\ \label{eq:MonotonicityAssumption}
y_{1}\geq y_{2}\Rightarrow f(x,y_{1})\leq f(x,y_{2})~\forall x\in \intcc{a,b}.
\end{align}  
\end{subequations}
These properties  resemble different contact models with linear and nonlinear springs. A particular instance of $f$ that will be investigated herein, which also satisfies assumptions \eqref{eq:Assumptions} (with $M=0$), is
\begin{equation}\label{eq:PWLCase}
  f(x,y)=-K (y-g(x))\mathbf{H}(y-g(x)),
\end{equation}
where $g:\intcc{a,b}\rightarrow \mathbb{R}$ is a  continuously differentiable function defining  a contact surface,  $K\in \mathbb{R}_{+}$ is a stiffness coefficient,  and $\mathbf{H}:\mathbb{R}\rightarrow \mathbb{R}$ is the Heaviside function. 

\section{Literature review}
\label{sec:LiteratureReview}
Fourth-order BVPs have been studied extensively in the literature. For example, Usmani \cite{Usmani79} studied the linear fourth-order BVP $z^{(4)}+\mathcal{F}(x)z=r(x)$, with BCs as in \eqref{eq:BCs}, showing solution uniqueness if $\mathcal{F}$ satisfies a prescribed one-sided bound\footnote{The mathematical analysis in \cite{Usmani79} was, in fact,  insufficient to prove the solution  uniqueness claim under the one-sided boundedness property, which necessitated introducing a refined  proof in  \cite{Yang88}.} Aftabizadeh \cite{Aftabizadeh86} investigated the  BVP 
$
    z^{(4)}=\mathcal{F}(x,z,z'')$, with BCs as in \eqref{eq:BCs},
proving solution existence  when  $\mathcal{F}$ is uniformly bounded by a constant.
Yang \cite{Yang88} has improved this result, showing that when  $\mathcal{F}$ is bounded by an affine function (linear growth condition) with parameters not exceeding  prescribed thresholds, then solutions to the considered BVP exist. This linear growth condition has been relaxed in the succeeding  works \cite{LiYang10,LiGao19}, where the BCs in \eqref{eq:BCs} are zero. Li and Liang \cite{LiLiang13} extended the  result of Yang \cite{Yang88} to cover BVPs with right-hand side of the form $\mathcal{F}(x,z,z',z'',z''')$ (where the BCs in \eqref{eq:BCs} are homogeneous). Quang A and Quy  \cite{QuangAQuy18} studied  the fully nonlinear BVP  $z^{(4)}=\mathcal{F}(x,z,z',z'',z''')$, with  the BCs in \eqref{eq:BCs} being homogeneous, showing solution existence and uniqueness within a rectangular  domain of $(x,z,z',z'',z''')$ if $\mathcal{F}$ is bounded, within that domain,  by a constant that depends on the rectangular domain and that $\mathcal{F}$ is Lipschitz within that domain. 

In our current study, we require the right-hand side $f$ to satisfy a monotonic property (equation \eqref{eq:MonotonicityAssumption}). There exist several analyses in the literature concerning fourth-order BVPs with monotonic right-hand sides. For instance, Agarwal \cite{Agarwal89} showed that the BVP $z^{(4)}=\mathcal{F}(x,z)$, with BCs different from those in \eqref{eq:BCs}, possesses at most one solution if $\mathcal{F}$ is non-increasing with respect to the second argument. Bai \cite{Bai00} studied the BVP $z^{(4)}=\mathcal{F}(x,z,z'')$, with a homogeneous version of the BCs in \eqref{eq:BCs}, showing that if the BVP possesses upper and lower solutions, where $\mathcal{F}$ satisfies some monotonic properties, then there exists at least one solution to the BVP. Han and Li \cite{HanLi07} studied  the BVP $z^{(4)}=\mathcal{F}(x,z)$, with a homogeneous version of \eqref{eq:BCs}, showing that if $\mathcal{F}$ is monotnonically increasing in the second argument, and that some  technical assumptions are full-filled, then the BVP possesses multiple solutions. Besides the mentioned works, the literature is rich in various investigations of fourth-order BVPs with different BCs  (see, e.g., \cite{FengJiGe09,LiChen19,QuangAHuong18,QuangAQuy17,CaballeroHarjaniSadarangani11}).   To the best of our knowledge, the  results in the literature can not be directly applied to address the existence and uniqueness of the BVP \eqref{eq:GeneralBVP} with the assumptions in \eqref{eq:Assumptions}; which, in addition to the biomechanical, electromechanical, and structural applications discussed in the Introduction, motivate the analysis of the BVP \eqref{eq:GeneralBVP} in this study.

\section{Existence and uniqueness results}
\label{sec:ExistenceUniquenss}
We begin our analysis of the BVP \eqref{eq:GeneralBVP} with the following existence result.
\begin{theorem}\label{thm:ExistenceContinuous}
    The BVP \eqref{eq:GeneralBVP} with $f$ satisfying 
assumptions \eqref{eq:Assumptions} possesses at least one solution.
\end{theorem}
\begin{proof}
The  BVP \eqref{eq:GeneralBVP} can be written in the equivalent integral form (see, e.g.,  \cite{RuyunJihuiShengmao97})
\begin{equation}
\label{eq:BVPIntegral}
w(x)=\bar{w}_{M}(x)+\int_{a}^{b}\tilde{G}(x,s) \left(\int_{a}^{b}\tilde{G}(s,t)f_{M}(t,w(t))\mathrm{d}t\right) \mathrm{d}s,~x\in \intcc{a,b},
\end{equation}
where  $\bar{w}_{M}:\intcc{a,b}\rightarrow \mathbb{R}$ is the fourth-order polynomial satisfying $\bar{w}^{(4)}=M$ and  the BCs in \eqref{eq:BCs},  $\tilde{G}:\intcc{a,b}\times \intcc{a,b}\rightarrow \mathbb{R}$ is the Green's function associated with the second-order BVP $y''=h,~y(a)=y(b)=0$, and 
\begin{equation}\label{eq:Fm}
f_{M}(\cdot,\cdot)\defas f(\cdot,\cdot)-M.   
\end{equation}
Note that $f_{M}$ is continuous, nonpositive, and monotonically decreasing in its second argument due to assumptions \eqref{eq:Assumptions}. The polynomial $\bar{w}_{M}$ is given by:
\begin{equation}\label{eq:wbar}
\begin{split} 
  \bar{w}_{M}(x)=&\alpha_{1}+\left(\frac{\alpha_{2}-\alpha_{1}}{b-a}-(\beta_{2}+2\beta_{1})\frac{b-a}{6}+\frac{M (b-a)^3}{24}\right)(x-a)+\frac{\beta_{1}}{2}(x-a)^2\\
  & +\left(\frac{\beta_{2}-\beta_{1}}{6(b-a)}-\frac{M(b-a)}{12}\right)(x-a)^3+\frac{M}{24}(x-a)^4, 
\end{split}
\end{equation}
and the Green's function $\tilde{G}$ is given by:
\begin{equation}\label{eq:TildeG}
  \tilde{G}(x,s)=\frac{-1}{b-a}\begin{cases} (b-x)(s-a),& a\leq s\leq x\leq b, \\ 
   (b-s)(x-a),& a\leq x\leq s\leq b. 
  \end{cases}  
\end{equation}
Note that $\tilde{G}(x,s)\leq0$ and $\abs{\tilde{G}(x,s)}\leq b-a$ for all $(x,s)\in \intcc{a,b}\times \intcc{a,b}$. By applying Fubini's theorem (see, e.g., \cite{RoydenFitzpatrick88,Tao11}), the iterated integral in \eqref{eq:BVPIntegral} can be rewritten as
\begin{equation}\label{eq:IteratedIntegral}
\int_{a}^{b} \left(\int_{a}^{b}\tilde{G}(x,s)\tilde{G}(s,t)\mathrm{d}s\right)f_{M}(t,w(t)) \mathrm{d}t=\int_{a}^{b}G(x,s)f_{M}(s,w(s))\mathrm{d}s,
\end{equation}
where
\begin{equation}\label{eq:G}
G(x,s)\defas \int_{a}^{b}\tilde{G}(x,t)\tilde{G}(t,s)\mathrm{d}t
\end{equation}
is the Green's function associated with the BVP \eqref{eq:GeneralBVP}. Note that $G$ is continuous and nonnegative (due to the nonpositivity of $\tilde{G}$). Using the definition of $\tilde{G}$ in \eqref{eq:TildeG}, it can be shown that 
$
\abs{\tilde{G}(x,s)-\tilde{G}(y,s)}\leq 3 \abs{x-y}
$
for all $x,y,s\in \intcc{a,b}$, which implies, using  the definition \eqref{eq:G}, that
$ 
\abs{{G}(x,s)-{G}(y,s)}\leq L_{G} \abs{x-y}
$
for all $x,y,s\in \intcc{a,b}$, where 
\begin{equation}\label{eq:LG}
   L_{G}\defas 3 \sup_{s\in \intcc{a,b}} \int_{a}^{b}\abs{\tilde{G}(t,s)}\mathrm{d}t\leq 3 (b-a)^2.
\end{equation}

Define the operator $\mathcal{L}:C\intcc{a,b}\rightarrow C^{4}\intcc{a,b}$ as
\begin{equation}\label{eq:OperatorL}
\mathcal{L}[v](x)\defas\bar{w}_{M}(x)+\int_{a}^{b}G(x,s)f_{M}(s,v(s))\mathrm{d}s.
\end{equation}
The integral definition \eqref{eq:BVPIntegral}, in addition to \eqref{eq:IteratedIntegral}, indicates that a solution to the BVP \eqref{eq:GeneralBVP} is a fixed point for the operator $\mathcal{L}$. We will show by means of Schauder's fixed-point theorem that $\mathcal{L}$ has at least one fixed point.  The nonnegativity of $G$ and the nonpositivity and monotonicity  of $f_{M}$ imply the following for any $p,q\in C\intcc{a,b}$:
\begin{align}
p\geq  q \Rightarrow \mathcal{L}[p]\leq \mathcal{L}[q],\\
 \mathcal{L}[p]\leq \bar{w}_{M}.
\end{align}
Therefore, for all $q\in C\intcc{a,b}$,
\begin{equation}\label{eq:InclusionProperty}
\mathcal{L}[\bar{w}_{M}]\leq q\leq \bar{w}_{M}\Rightarrow \mathcal{L}[\bar{w}_{M}]\leq \mathcal{L}[q]\leq \bar{w}_{M}.
\end{equation}
By taking this fact into consideration, define 
$$
D_{\eta}\defas\{q\in C\intcc{a,b},~ \mathcal{L}[\bar{w}_{M}]\leq q\leq \bar{w}_{M},~q\in\mathrm{Lip}_{\eta}\intcc{a,b} \},
$$
where
$$
\eta:=\sup_{x \in \intcc{a,b}} \abs{\bar{w}_{M}'(x)}+ L_{G} \sup_{x\in \intcc{a,b}} |f_{M}(x,\bar{w}_{M}(x))| (b-a)
$$
and $L_{G}$ is defined in \eqref{eq:LG}.
By straightforward analysis, it can be shown that $D_{\eta}$ is convex and closed. Moreover, using the Arzela–Ascoli theorem (see, e.g., \cite{Zeidler12}), $D_{\eta}$ is a  relatively  compact subset of the Banach space $C\intcc{a,b}$ (hence, $D_{\eta}$ is compact).

Now, we claim that 
\begin{equation}\label{eq:SelfMap}
\mathcal{L}:D_{\eta}\rightarrow D_{\eta}.
\end{equation}
To prove this, we first note that, using \eqref{eq:InclusionProperty}, 
$
q\in D_{\eta}\Rightarrow \mathcal{L}[\bar{w}_{M}]\leq \mathcal{L}[q]\leq \bar{w}_{M}.
$
It is now left  to show that for  $q\in D_{\eta}$, $\mathcal{L}[q]\in\mathrm{Lip}_{\eta}\intcc{a,b}$. Let $q\in D_{\eta}$ and  $x,y\in \intcc{a,b}$. Then,
\begin{align*}
\abs{\mathcal{L}[q](y)-\mathcal{L}[q](x)}=& \abs{ \bar{w}_{M}(y)-\bar{w}_{M}(x)+ \int_{a}^{b}\left(G(y,s)-G(x,s)\right)f_{M}(s,q(s))\mathrm{ds}}\\
& \leq \abs{\bar{w}_{M}(y)-\bar{w}_{M}(x)}+\int_{a}^{b}\abs{G(y,s)-G(x,s)}\abs{f_{M}(s,q(s))}\mathrm{d}s. 
\end{align*}
By assumptions \eqref{eq:NegativityAssumption} and \eqref{eq:MonotonicityAssumption}, and the definitions of $D_\eta$ and $f_{M}$, we have
$
\abs{f_{M}(s,q(s))}\leq \abs{f_{M}(s,\bar{w}_{M}(s))}~\forall s\in \intcc{a,b}.
$
Moreover, by the mean value theorem,
$
    \bar{w}_{M}(y)-\bar{w}_{M}(x)=w_{M}^{'}(z_{1})(y-x),
$
for some $z_{1}\in [x,y]$, and as we showed previously, 
$
\abs{G(y,s)-G(x,s)} \leq L_{G}\abs{y-x}.
$
Therefore,
$$
\abs{\mathcal{L}[q](y)-\mathcal{L}[q](x)}
\leq \sup_{z \in \intcc{a,b}} |\bar{w}_{M}^{'}(z)| |y-x|+\int_{a}^{b}L_{G} \sup_{z \in \intcc{a,b}} |f_{M}(z,\bar{w}_{M}(z))|\mathrm{d}s |y-x|\leq \eta |y-x|,
$$
and that proves \eqref{eq:SelfMap}.  

Next, we show that    $\mathcal{L}$ is continuous on $D_{\eta}$. Define 
$
R_{D_{\eta}}=[\min_{x\in \intcc{a,b}}\mathcal{L}[\bar{w}_{M}](x), \max_{x\in \intcc{a,b}}\bar{w}_{M}(x)].
$
The  definitions of $D_ {\eta}$ and $R_{D_{\eta}}$   indicate that 
$
q\in D_{\eta}\Rightarrow q(x)\in R_{D_{\eta}}~\forall x\in \intcc{a,b}.
$
Fix $p\in D_{\eta}$ and let $\epsilon>0$ be arbitrary. Then, using the  uniform continuity of $f_{M}$ on the compact set $\intcc{a,b}\times R_{D_{\eta}}$, there exists $\tau>0$ (that depends on $\epsilon$ only) such that for any $q\in  D_{\eta}$ satisfying
$
\norm{p-q}_{C\intcc{a,b}}\leq \tau$ , $$ \sup_{s\in \intcc{a,b}}|f_{M}(s,p(s))-f_{M}(s,q(s))|< \frac{\epsilon}{\sup_{x\in \intcc{a,b}}\int_{a}^{b}G(x,s)\mathrm{d}s},$$ implying
 \begin{align*}
|\mathcal{L}[p](x)-\mathcal{L}[q](x)|&\leq \int_{a}^{b}\abs{G(x,s)}\abs{f_{M}(s,p(s))-f_{M}(s,q(s))}\mathrm{d}s\\
 &\leq \sup_{t\in \intcc{a,b}}\int_{a}^{b}G(t,s)\mathrm{d}s \sup_{s\in \intcc{a,b}}|f_{M}(s,p(s))-f_{M}(s,q(s))|< \epsilon~\forall x\in \intcc{a,b} 
 \end{align*}
(i.e., $\norm{\mathcal{L}[p]-\mathcal{L}[q]}_{C\intcc{a,b}}< \epsilon$).  Therefore, by Schauder's fixed-point theorem \cite{Zeidler12}, $\mathcal{L}$ has at least one fixed point in $D_\eta$, implying solution existence to the  BVP \eqref{eq:GeneralBVP}.
\end{proof}
Next, we show that our BVP  possesses at most one solution.

\begin{theorem} \label{Thm:AtMostUniqueness}
 The BVP \eqref{eq:GeneralBVP}, with $f$ satisfying \eqref{eq:ContinuityAssumption} and \eqref{eq:MonotonicityAssumption}, possesses at most one solution.
 \end{theorem}
\begin{proof}
    Let $\phi,\psi\in C^{4}\intcc{a,b}$ be two solutions to the BVP \eqref{eq:GeneralBVP}, where $f$ satisfies \eqref{eq:ContinuityAssumption} and \eqref{eq:MonotonicityAssumption}, and define the error function $E\defas \phi-\psi\in C^{4}[a,b]$. Then, ${E}$ satisfies
${E}^{(4)}=f(x,\phi)-f(x,\psi)$, $
E(a)=E(b)=E''(a)=E''(b)=0.$
Following the approach in \cite{Agarwal89} and by multiplying $E^{(4)}$ by $E$, we get
$
E(x) E^{(4)}(x)=(\phi(x)-\psi(x))(f(x,\phi(x))-f(x,\psi(x)))\leq 0~\forall x\in \intcc{a,b}
$
due to the monotonicity assumption \eqref{eq:MonotonicityAssumption}. Integrating the both sides of the  inequality,  using the homogeneous BCs of ${E}$, and utilizing integration of part yield 
$
\int_{a}^{b}E(s) E^{(4)}(s)\mathrm{d}s=\cancel{\left[E(x)E^{'''}(x)\right]_{x=a}^{x=b}}-\int_{a}^{b}E^{'}(s) E^{'''}(s)\mathrm{d}s
=-\cancel{\left[E'(x)E^{''}(x)\right]_{x=a}^{x=b}}+ \int_{a}^{b}(E^{''}(s))^2\mathrm{d}s\leq 0.
$
Hence, 
$
E^{''}=0.
$
This implies, using the fact that $E(a)=E(b)=0$, that $E=0$. 
\end{proof} 

Combining Theorems \ref{thm:ExistenceContinuous} and \ref{Thm:AtMostUniqueness}, we have: 
\begin{theorem} \label{Thm:MainExistenceUniquenessResult}
 The BVP \eqref{eq:GeneralBVP} with $f$ satisfying  \eqref{eq:Assumptions} possesses a unique solution.
 \end{theorem}

\section{Numerical solutions}
\label{sec:NumericalApproach}
In this section, we  analyze the numerical solutions to the BVP \eqref{eq:GeneralBVP}, where we adapt  
the finite difference
discretization scheme presented in  \cite{usmani1975numerical}, which was devised for linear fourth-order BVPs. The literature is rich in various numerical approaches for fourth-order BVPs, including  Adomian decomposition  method  \cite{wazwaz2002numerical}, wavelet-based methods  \cite{ali2011numerical}, spline-based methods \cite{zahra2011finite}, differential transform methods \cite{erturk2007comparing}, and Galerkin methods \cite{hajji2008numerical}. However, our adoption of the finite difference method herein is motivated by the facts that the resulting discrete equation possesses, like the original continuous-space problem, an existence and uniqueness property (Theorem \ref{thm:ExistenceUniquenessDiscrete}), and that in some cases, the solutions of the discrete equation can be approximated efficiently via means of fixed-point iterations (Theorem \ref{thm:FixedPointIterations}), where solutions to the discrete equation can be shown to converge to solutions of the original BVP (Theorem \ref{thm:Convergence}). Such aspects are typically difficult to achieve with the numerical methods in the literature, especially with the weak assumptions imposed in this work.

 Let 
$N\in \mathbb{N}$, $h=(b-a)/(N+1)$, $x_{0}=a$,  $x_{i}=a+ ih,~i\in [1;N]$, and $x_{N+1}=b$. Moreover, let $\mathbf{w}_{i},~i\in [0;N+1]$, denote the approximate  values of the solution $w$ to \eqref{eq:GeneralBVP}, with $f$ satisfying \eqref{eq:Assumptions}, at $x_{i}$ ($\mathbf{w}_{i}\approx w(x_{i})$),  corresponding to solutions of a discrete version of the BVP \eqref{eq:GeneralBVP} obtained using  finite difference (forward and backward difference approximations of second-order derivatives at the boundary points $a$ and $b$, respectively, and central difference approximations of the fourth-order derivatives). In other words, $\{\mathbf{w}_{i}\}_{i=0}^{N+1}$ satisfies: $\mathbf{w}_{0}=\alpha_{1}$, $
2\mathbf{w}_{0} -5\mathbf{w}_{1}+4\mathbf{w}_{2}-\mathbf{w}_{3} =\beta_{1}h^{2}-h^{4}f(x_{1},\mathbf{w}_{1})+\frac{1}{12}h^{4}f(x_{0},\mathbf{w}_{0})$, $
\mathbf{w}_{i-2}-4\mathbf{w}_{i-1}+6\mathbf{w}_{i}-4\mathbf{w}_{i+1}+\mathbf{w}_{i+2}=h^{4
}f(x_{i},\mathbf{w}_{i}),~i\in [2;N-2]$, 
$2\mathbf{w}_{N+1} -5\mathbf{w}_{N}+4\mathbf{w}_{N-1}-\mathbf{w}_{N-2}=\beta_{2}h^{2}-h^{4}f(x_{N},\mathbf{w}_{N})+\frac{1}{12}h^{4}f(x_{N+1},\mathbf{w}_{N+1})$,
$\mathbf{w}_{N+1}=\alpha_{2}$.
This set of equations can be rewritten as
$
5\mathbf{w}_{1}-4\mathbf{w}_{2}+\mathbf{w}_{3}=-\beta_{1}h^{2}+2 \alpha_{1} +h^{4}f(x_{1},\mathbf{w}_{1})-\frac{1}{12}h^{4}f(x_{0},\alpha_{1})$, 
$-4\mathbf{w}_{1}+6\mathbf{w}_{2}-4\mathbf{w}_{3}+\mathbf{w}_{4}=h^{4
}f(x_{2},\mathbf{w}_{2})-\alpha_{1}$,
 $\mathbf{w}_{i-2}-4\mathbf{w}_{i-1}+6\mathbf{w}_{i}-4\mathbf{w}_{i+1}+\mathbf{w}_{i+2}=h^{4
}f(x_{i},\mathbf{w}_{i}),~i\in [3;N-2]$,
$\mathbf{w}_{N-3}-4\mathbf{w}_{N-2}+6\mathbf{w}_{N-1}-4\mathbf{w}_{N}=h^{4
}f(x_{N-1},\mathbf{w}_{N-1})-\alpha_{2}$,
 $5\mathbf{w}_{N}-4\mathbf{w}_{N-1}+\mathbf{w}_{N-2}=2\alpha_{2}-\beta_{2}h^{2}+h^{4}f(x_{N},\mathbf{w}_{N})-\frac{1}{12}h^{4}f(x_{N+1},\alpha_{2})$, which, in the vector-matrix form,  can be written as 
\begin{equation}\label{eq:DiscreteEquation}
\mathcal{A}\mathcal{W}=\bar{\mathcal{B}}+h^{4}\left(\mathcal{M}+{\mathcal{F}}_{M}(\mathcal{W})\right),
\end{equation}
where $\mathcal{W}=[\mathbf{w}_{1},\cdots,\mathbf{w}_{N}]^{\intercal}\in \mathbb{R}^{N}$,  $\mathcal{A}\in \mathbb{R}^{N\times N}$ and $\bar{\mathcal{B}}\in \mathbb{R}^{N}$ are defined as
$$
\mathcal{A}=\left(\begin{array}{ccccccc}
    5 & -4 & 1 & & & &  \\
      -4 &6 &-4 &1& & & \\
     1 &-4 &6&-4&1 & &\\
     \ddots&\ddots &\ddots &\ddots &\ddots &\ddots &\\
     & & &-4 & 6& -4&1\\
     & & & 1&-4 &6 &-4\\
     & & & &1 &-4 &5\\
\end{array} \right),~[\bar{\mathcal{B}}]_{i}=\begin{cases} -\beta_{1}h^{2}+2 \alpha_{1}-\frac{1}{12}h^{4}f(x_{0},\alpha_{1}),& i=1,\\
-\alpha_{1},& i=2,\\
0,& i\in [3;N-2],\\
-\alpha_{2},& i=N-1,\\
-\beta_{2}h^{2}+2 \alpha_{2}-\frac{1}{12}h^{4}f(x_{N+1},\alpha_{2}),& i=N,
\end{cases}
$$
and $\mathcal{M}\in \mathbb{R}^{N}$ and ${\mathcal{F}}_{M}\colon  \mathbb{R}^{N}\rightarrow  \mathbb{R}^{N}$ are given by $[\mathcal{M}]_{i}=M$,   
$ 
[\mathcal{F}_{M}(Y)]_{i}=
f(x_{i},{Y}_{i})-M,~ i\in [1;N],~ Y\in \mathbb{R}^{N}.$
Matrix $\mathcal{A}$ is symmetric positive definite (hence, invertible) \cite{usmani1975numerical}; therefore,   
\begin{equation}\label{eq:DiscreteEquationGeneral}
\mathcal{W}=\mathcal{A}^{-1}\bar{\mathcal{B}}+\mathcal{A}^{-1}h^{4}\left(\mathcal{M}+{\mathcal{F}}_{M}(\mathcal{W})\right).
\end{equation}
In other words, $\mathcal{W}$  is a fixed point for the function $\mathcal{R}\colon \mathbb{R}^{N}\rightarrow \mathbb{R}^{N}$, defined as
\begin{equation}\label{eq:R}
\mathcal{R}(\mathcal{V})\defas\mathcal{A}^{-1}\bar{\mathcal{B}}+\mathcal{A}^{-1}h^{4}\left(\mathcal{M}+{\mathcal{F}}_{M}(\mathcal{V})\right),~\mathcal{V}\in \mathbb{R}^{N}.
\end{equation}
By following a  reasoning similar to that used in deducing Theorem \ref{Thm:MainExistenceUniquenessResult}, we obtain:
\begin{theorem}\label{thm:ExistenceUniquenessDiscrete}
    The discrete BVP given by equation \eqref{eq:DiscreteEquationGeneral} has a unique solution, assuming $f$ satisfies \eqref{eq:Assumptions}.
\end{theorem}
\begin{proof}
Let 
$
\mathcal{D}=\{x\in \mathbb{R}^{N},\mathcal{R}(\mathcal{A}^{-1}(\bar{\mathcal{B}}+h^{4}\mathcal{M})) \leq x \leq \mathcal{A}^{-1}(\bar{\mathcal{B}}+h^{4}\mathcal{M})\}.
$
Note that $\mathcal{D}$ nonempty, due to the nonpositivity of $f_{M}$,  convex, and compact. Moreover,  $\mathcal{R}:\mathcal{D}\rightarrow \mathcal{D}$. Therefore, by Schauder's fixed-point theorem, $\mathcal{R}$ has at least one fixed point. Let $\mathcal{V}$ and $\mathcal{W}$ be two fixed points of $\mathcal{R}$. Then, we have 
$
\mathcal{V}-\mathcal{W}=\mathcal{A}^{-1}h^{4}({\mathcal{F}}_{M}(\mathcal{V})-{\mathcal{F}}_{M}(\mathcal{W})).
$
Consequently,
$$
({\mathcal{F}}_{M}(\mathcal{V})-{\mathcal{F}}_{M}(\mathcal{W}))^{\intercal}(\mathcal{V}-\mathcal{W})= ({\mathcal{F}}_{M}(\mathcal{V})-{\mathcal{F}}_{M}(\mathcal{W}))^{\intercal}\mathcal{A}^{-1}h^{4}({\mathcal{F}}_{M}(\mathcal{V})-{\mathcal{F}}_{M}(\mathcal{W})).
$$
We have $({\mathcal{F}}_{M}(\mathcal{V})-{\mathcal{F}}_{M}(\mathcal{W}))^{\intercal}(\mathcal{V}-\mathcal{W})\leq 0$ due to the monotonicity of $f_{M}$. On the other hand, and due to the positive definiteness of $\mathcal{A}^{-1}$, we have  $({\mathcal{F}}_{M}(\mathcal{V})-{\mathcal{F}}_{M}(\mathcal{W}))^{\intercal}\mathcal{A}^{-1}h^{4}({\mathcal{F}}_{M}(\mathcal{V})-{\mathcal{F}}_{M}(\mathcal{W}))\geq 0$, with equality holding only if ${\mathcal{F}}_{M}(\mathcal{V})-{\mathcal{F}}_{M}(\mathcal{W})=0$; hence, ${\mathcal{F}}_{M}(\mathcal{V})={\mathcal{F}}_{M}(\mathcal{W})$.  Consequently, and using equation \eqref{eq:DiscreteEquationGeneral}, we have $\mathcal{V}=\mathcal{W}$.  
\end{proof}
In the next section, we show how the unique solution to \eqref{eq:DiscreteEquationGeneral} for the piecewise linear instance of $f$ given by equation \eqref{eq:PWLCase} can be attained by means of fixed-point iterations.
\section{Numerical solution for the piecewise linear case}
   Recall the discrete equation \eqref{eq:DiscreteEquation}\label{sec:NumericalPWL} and the piecewise linear instance of $f$ given by equation \eqref{eq:PWLCase}.  Let $\mathbf{g}_{i},~i\in [0;N+1]$, denote the   values of $g$ at $x_{i}$ ($\mathbf{g}_{i}= g(x_{i})$), and $\mathcal{G}=[\mathbf{g}_{1},\cdots,\mathbf{g}_{N}]^{\intercal}\in \mathbb{R}^{N}$. For that case, the discrete equation \eqref{eq:DiscreteEquation} can be written as
$
\mathcal{A}\mathcal{W}=\bar{\mathcal{B}}-h^{4}K(\mathcal{W}-\mathcal{G})\mathbf{H}(\mathcal{W}-\mathcal{G}).
$
Using the facts that  $\mathbf{H}(x)=(1+\mathrm{sign}(x))/2$ ($\mathrm{sign}$ is the signum function), and $x\times \mathrm{sign}(x)=|x|$, we  get
\begin{equation}\label{eq:AVEInitial}
(\mathcal{A}+\frac{h^{4}}{2}K\id)\mathcal{W}=\bar{\mathcal{B}}+\frac{h^{4}K}{2}\mathcal{G}-\frac{h^{4}K}{2}|\mathcal{W}-\mathcal{G}|,
\end{equation}
Using the substitution $\mathcal{Z}=\mathcal{W}-\mathcal{G}$, the above equation can be rewritten as
\begin{equation}\label{eq:AVE}
(\mathcal{A}+\frac{h^{4}}{2}K\id)\mathcal{Z}=\bar{\mathcal{B}}-\mathcal{A}\mathcal{G}-\frac{h^{4}K}{2}|\mathcal{Z}|,
\end{equation}
Equation \eqref{eq:AVE} is an instance of the so-called  absolute value equation, which has been investigated extensively in the literature due to its relevance in mathematical programming problems, with several existence and uniqueness results and proposed solution procedures (see, e.g.,  \cite{MangasarianMeyer06,Mezzadri20,RohnHooshyarbakhshFarhadsefat14,WuLi18, WuLi20,RadonsRump22}). Herein, we aim to solve equation \eqref{eq:AVEInitial} by means of fixed-point iterations. 
By utilizing the invertibility of $\mathcal{A}+\frac{h^{4}}{2}K\id$, for all $K\in \mathbb{R}_{+}$,\footnote{The sum of a positive definite  matrix and a semi positive definite matrix is positive definite; hence, invertible.} equation \eqref{eq:AVEInitial}  can equivalently be rewritten as
$  
\mathcal{W}=(\mathcal{A}+\frac{h^{4}}{2}K\id)^{-1}(\bar{\mathcal{B}}+\frac{h^{4}}{2}K\mathcal{G})-\frac{h^{4}K}{2}(\mathcal{A}+\frac{h^{4}}{2}K\id)^{-1}|\mathcal{W}-\mathcal{G}|.
$
In other words, $\mathcal{W}$ is a fixed point for the operator 
\begin{equation}\label{eq:ContractionMapping}
\mathcal{T}(X)\defas (\mathcal{A}+\frac{h^{4}}{2}K\id)^{-1}(\bar{\mathcal{B}}+\frac{h^{4}}{2}K\mathcal{G})-\frac{h^{4}K}{2}(\mathcal{A}+\frac{h^{4}}{2}K\id)^{-1}|X-\mathcal{G}|,~ X\in \mathbb{R}^{N}.
 \end{equation}
Below, we show the contractive property of $\mathcal{T}$.
\begin{lemma}\label{lem:Contractive}
For all $K\in \mathbb{R}_{+}$, the mapping $\mathcal{T}$ in \eqref{eq:ContractionMapping} is contractive. 
\end{lemma}
\begin{proof}
 For any $X,Y\in \mathbb{R}^{N}$, and any $p$-norm, $p\in \intcc{1,\infty}$,
$  
\norm{\mathcal{T}(X)-\mathcal{T}(Y)}_{p}\leq 
\norm{\frac{h^{4}K}{2}(\mathcal{A}+\frac{h^{4}}{2}K\id)^{-1}}_{p}\norm{\abs{X-\mathcal{G}}-\abs{Y-\mathcal{G}}}_{p}\leq \norm{\frac{h^{4}K}{2}(\mathcal{A}+\frac{h^{4}}{2}K\id)^{-1}}_{p}\norm{X-Y}_{p}. 
$
Therefore, it is sufficient to show that  $
\norm{\frac{h^{4}K}{2}(\mathcal{A}+\frac{h^{4}}{2}K\id)^{-1}}_{2}<1.
    $ 
    Note that $\frac{h^{4}K}{2}(\mathcal{A}+\frac{h^{4}}{2}K\id)^{-1}$ is symmetric; hence its 2-norm is equivalent to its largest eigenvalue. The eigenvalues of $\mathcal{A}$ are (see \cite{Usmani79,usmani1975numerical})
$
\lambda_{i}(\mathcal{A})=16\sin^{4}\left(\frac{i \pi}{2(N+1)}\right),~i\in [1;N].
$
Hence, using Weyl's inequality \cite{HornJohnson12},
$
\lambda_{i}\left(\mathcal{A}+\frac{h^{4}K}{2}\id\right)=\frac{h^{4}K}{2}+16\sin^{4}\left(\frac{i \pi}{2(N+1)}\right),~i\in \intcc{1;N}.
$
Consequently,
\begin{align*}
 \lambda_{i}(\frac{h^{4}K}{2}(\mathcal{A}+\frac{h^{4}}{2}K\id)^{-1})&=\frac{\frac{h^{4}K}{2}}{\frac{h^{4}K}{2}+16\sin^{4}\left(\frac{(N+1-i) \pi}{2(N+1)}\right)}\leq \frac{(b-a)^4K}{(b-a)^4K+32} <1,~i\in \intcc{1;N},
\end{align*} 
where the  inequality follows from the fact that 
$
x\leq \sin({\pi}x/2),~x\in \intcc{0,1},
$
and the the fact that  $h=(b-a)/(N+1)$.
\end{proof}  
As a consequence of the above lemma, and using the contraction mapping theorem, we have
\begin{theorem}\label{thm:FixedPointIterations}
    For any $K\in\mathbb{R}_{+}$,  the  
fixed-point iterations 
    $$
W^{j}=(\mathcal{A}+\frac{h^{4}}{2}K\id)^{-1}(\bar{\mathcal{B}}+\frac{h^{4}}{2}K\mathcal{G})-\frac{h^{4}K}{2}(\mathcal{A}+\frac{h^{4}}{2}K\id)^{-1}|W^{j-1}-\mathcal{G}|,~j\in \mathbb{N}
,$$
where $W^{0}\in \mathbb{R}^{N}$ is arbitrary, converge to the unique solution of equation \eqref{eq:AVEInitial}, denoted $W^{\ast}$, where
$$
\norm{W^{j}-W^{\ast}}_{2}\leq \mathcal{C}^{j}\frac{1}{1-\mathcal{C}}\norm{W^{1}-W^{0}}_{2},~j\in \mathbb{N},~\mathcal{C}\defas\frac{(b-a)^4K}{(b-a)^4K+32}.
$$

\end{theorem}

\subsection{Convergence analysis}
\label{sec:Convergence}
In this section, we show the convergence of the  solution of equation \eqref{eq:AVEInitial} to the unique solution of  the BVP \eqref{eq:GeneralBVP}, with $f$ as in \eqref{eq:PWLCase}. 
\begin{theorem}\label{thm:Convergence}
Let $\mathcal{W}^{\mathrm{true}}$ be the vector of exact values of the unique solution $w$ to the BVP \eqref{eq:GeneralBVP} with $f$ as in \eqref{eq:PWLCase} at the nodal points ${x_{1}},\ldots,x_{N}$, and let $\mathcal{W}$ be the unique solution to the discrete equation \eqref{eq:AVEInitial}. If $K\in \mathbb{R}_{+}$, then there exists a positive constant $\tilde{C}$ (independent of N), such that 
$
\norm{\mathcal{W}^{\mathrm{true}}-\mathcal{W}}_{2} \leq \tilde{C} \sqrt{h}.
$
\end{theorem}
\begin{proof}
$\mathcal{W}^{\mathrm{true}}$ satisfies
\begin{equation}\label{eq:AVETrue}
(\mathcal{A}+\frac{h^{4}}{2}K\id)\mathcal{W}^{\mathrm{true}}=\bar{\mathcal{B}}+\frac{h^{4}K}{2}\mathcal{G}-\frac{h^{4}K}{2}|\mathcal{W}^{\mathrm{true}}-\mathcal{G}|+\mathcal{E}_{N},
\end{equation}
where $\mathcal{E}_{N}\in \mathbb{R}^{N}$ consists of the remainder terms resulting from the finite difference approximations. In particular, $\mathcal{E}_{N}$ is given explicitly as
$$
[\mathcal{E}_{N}]_{i}=\begin{cases}
\left(-h^{4}\int_{a}^{a+h}+5\int_{a}^{a+h}\frac{(a+h-s)^{4}}{4!}-4\int_{a}^{a+2h}\frac{(a+2h-s)^{4}}{4!}
+\int_{a}^{a+3h}\frac{(a+3h-s)^{4}}{4!}\right)w^{(5)}(s)\mathrm{d}s,~i=1,\\
\left(-\int^{x_{i}}_{x_{i}-2h}\frac{(x_{i}-2h-s)^{4}}{4!}+4\int^{x_{i}}_{x_{i}-h}\frac{(x_{i}-h-s)^{4}}{4!}-4\int_{x_{i}}^{x_{i}+h}\frac{(x_{i}+h-s)^{4}}{4!}+\int_{x_{i}}^{x_{i}+2h}\frac{(x_{i}+2h-s)^{4}}{4!}\right)w^{(5)}(s)\mathrm{d}s,~i\in \intcc{2;N-1},
\\
\left(h^{4}\int^{b}_{b-h}-5\int^{b}_{b-h}\frac{(b-h-s)^{4}}{4!}+4\int^{b}_{b-2h}\frac{(b-2h-s)^{4}}{4!}
-\int^{b}_{b-3h}\frac{(b-3h-s)^{4}}{4!}\right)w^{(5)}(s)\mathrm{d}s,~i=N.
\end{cases}
$$
Note that $w^{(4)}$ is absolutely continuous and equal to $-K(w-g)\mathbf{H}(w-g)$ with an integrable derivative  $w^{(5)}$  given by 
$w^{(5)}=-K (w'-g')(1+\mathrm{sign}(w-g)).$
Combining equations \eqref{eq:AVEInitial} and \eqref{eq:AVETrue} yields 
$$
(\mathcal{A}+\frac{h^{4}}{2}K\id)\left(\mathcal{W}^{\mathrm{true}}-\mathcal{W}\right)=-\frac{h^{4}K}{2}\left(\abs{\mathcal{W}^{\mathrm{true}}-\mathcal{G}}-\abs{\mathcal{W}-\mathcal{G}}\right)
+\mathcal{E}_{N}
$$
or
$$
\mathcal{W}^{\mathrm{true}}-\mathcal{W}=-\frac{h^{4}K}{2}(\mathcal{A}+\frac{h^{4}}{2}K\id)^{-1}\left(\abs{\mathcal{W}^{\mathrm{true}}-\mathcal{G}}-\abs{\mathcal{W}-\mathcal{G}}\right)
+(\mathcal{A}+\frac{h^{4}}{2}K\id)^{-1}\mathcal{E}_{N}.
$$
Consequently, 
$
\norm{\mathcal{W}^{\mathrm{true}}-\mathcal{W}}_{2} \leq  \norm{\frac{h^{4}K}{2}(\mathcal{A}+\frac{h^{4}}{2}K\id)^{-1}}_{2}\norm{\mathcal{W}^{\mathrm{true}}-\mathcal{W}  }_{2}+\norm{(\mathcal{A}+\frac{h^{4}}{2}K\id)^{-1}}_{2} \norm{\mathcal{E}_{N}}_{2}.
$
Using the definitions of $\mathcal{E}_{N}$ and $w^{(5)}$ above, it can be shown with straightforward calculations  that 
$
\norm{ \mathcal{E}_{N}}_{\infty} \leq C h^{5}
$
for some positive constant $C$ independent of $N$. Therefore, using the fact that $
\|x\|_{2}\leq \sqrt{n}\|x\|_{\infty}~\forall x\in \mathbb{R}^{n},
    $
$
\norm{ \mathcal{E}_{N}}_{2}\leq \bar{C}h^{4.5
}
$
for some positive constant $\bar{C}$ independent of $N$.
 Recall from the proof of Lemma \ref{lem:Contractive} that $
\norm{\frac{h^{4}K}{2}(\mathcal{A}+\frac{h^{4}}{2}K\id)^{-1}}_{2} \leq {(b-a)^4K}/((b-a)^4K+32)<1,
 $
which implies that 
$
\norm{(\mathcal{A}+\frac{h^{4}}{2}K\id)^{-1}}_{2}\leq {2(b-a)^4}/(((b-a)^4K+32)h^{4}).
$
Therefore, 
$$ 
\norm{\mathcal{W}^{\mathrm{true}}-\mathcal{W}}_{2} \leq\frac{(b-a)^4K}{(b-a)^4K+32}\norm{\mathcal{W}^{\mathrm{true}}-\mathcal{W}  }_{2}+\bar{C}\frac{2(b-a)^4}{((b-a)^4K+32)h^{4}}h^{4.5}
$$
implying $\|\mathcal{W}^{\mathrm{true}}-\mathcal{W}  \|_{2} \leq \tilde{C}h^{0.5}$ for some positive constant  
$\tilde{C}$ independent of $N$.
\end{proof}
\subsection{Illustrative example}
\label{sec:Example}
In this section, we illustrate the approximate solutions, obtained through fixed-point computations. Consider the BVP \eqref{eq:GeneralBVP} over the interval $\intcc{a,b}=\intcc{0,1}$, where $f$ is as in \eqref{eq:PWLCase}, $K=10^4$, $g(x)=x/2$, $w(0)=w(1)=0$ and $w''(0)=w''(1)=-20$. As discussed in the Introduction, such a BVP corresponds to the static configuration of VFs with incorporated contact forces.
Herein, we discretize the problem as in Section \ref{sec:NumericalApproach}, with different values of the discretization parameter $N$, and adopt the fixed-point iterations in Theorem \ref{thm:FixedPointIterations}, with different number of iterations (and zero initial guess). 
\begin{figure}[htpb]
    \centering
\includegraphics[width=0.45\linewidth]{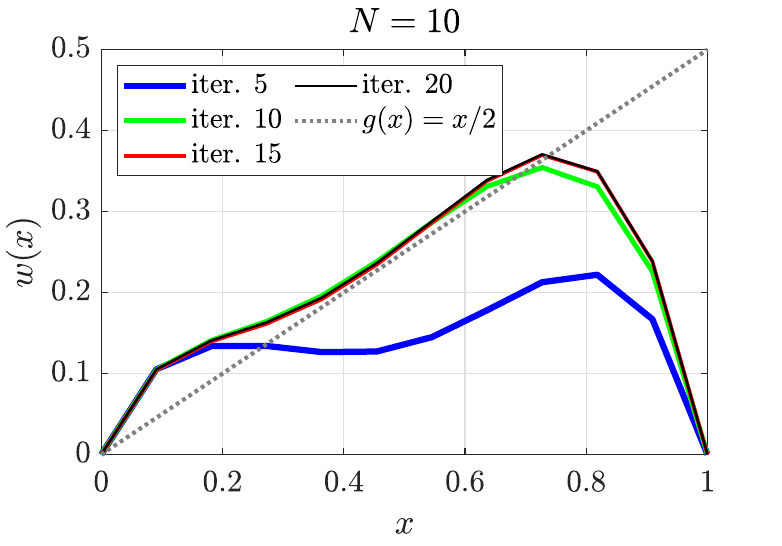}\includegraphics[width=0.45\linewidth]{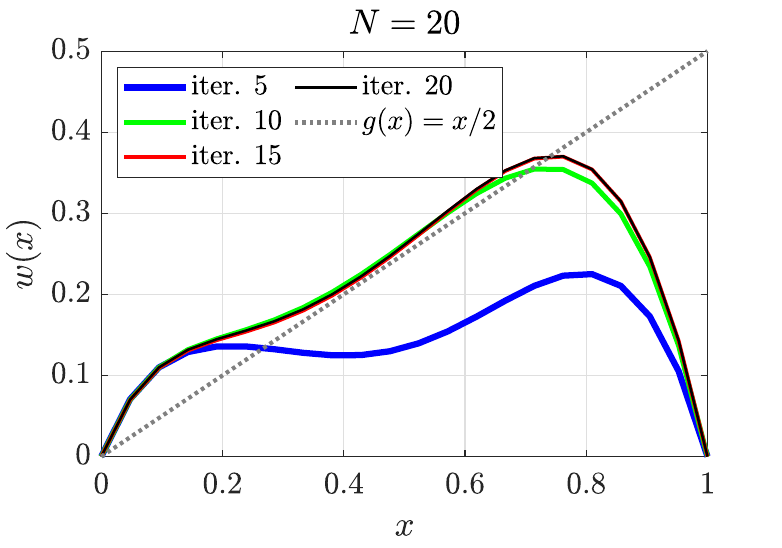}
\includegraphics[width=0.45\linewidth]{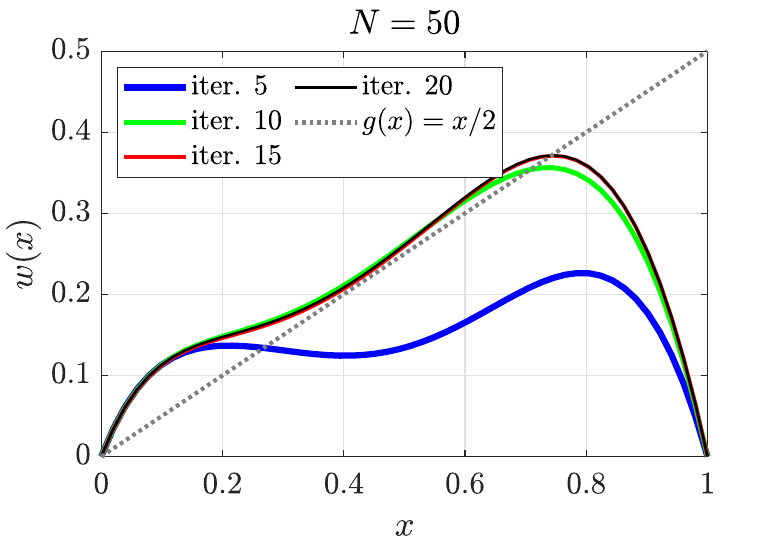}\includegraphics[width=0.45\linewidth]{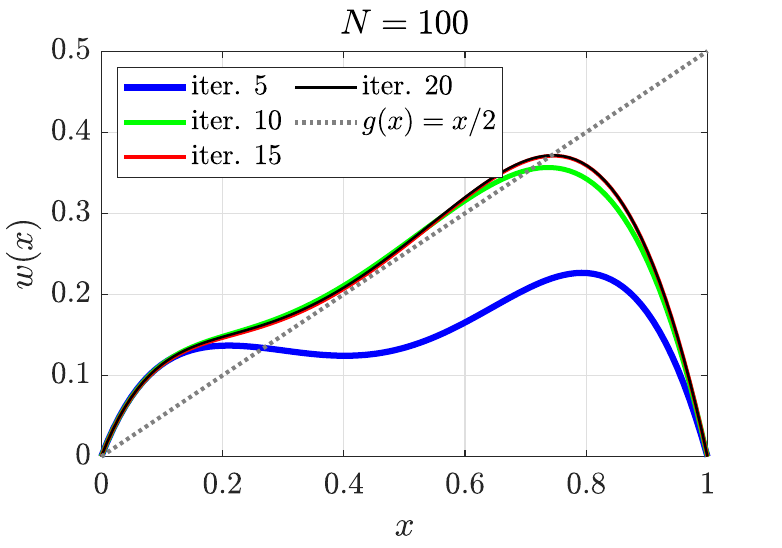}
    \caption{Numerical solutions to the BVP in Section \ref{sec:Example} with different values of the discretization parameter $N$ and different iterations of the fixed-point computations from Theorem \ref{thm:FixedPointIterations}. }
    \label{fig:Simulations}
\end{figure}
Figure \ref{fig:Simulations} shows that the fixed-point computations (almost) converge within the first 20 iterations for all the considered values of $N$. In addition, the figure shows that for the 20\textsuperscript{th} iteration  and different values of $N$, the numerical solutions almost coincide, especially for $N=50$  and $N=100$, which further supports the convergence result in Section \ref{sec:Convergence}.

\section{Conclusion}
\label{sec:Conclusion}
In this paper, we studied a fourth-order boundary value problem with a bounded above and monotonically decreasing right-hand side, showing solution existence and uniqueness for both the original problem and its discrete counterpart obtained through finite difference. We then focused on a particular piecewise linear instance of the right-hand side, showing that the solution to the discrete equation can be obtained through fixed-point iterations and that the solution to the discrete equation converges to the solution of the continuous problem. We illustrated the effectiveness of the fixed-point iterations through a numerical example. 

Typical voice production incorporates dynamic VF vibrations and time-varying contact forces (microswitches and beams on elastic foundations also often incorporate dynamic vibrations with contact forces). It would be of interest to  extend the techniques covered herein to explore dynamic Euler-Bernoulli beam models with contract forces. Moreover, extensions of the analysis presented herein can be considered in future works to cover more sophisticated beam (e.g., Timoshenko model \cite{Timoshenko21}) and plate models. 

 \bibliographystyle{ieeetr}

\end{document}